\documentclass[12pt]{article}
\usepackage{amsmath,amsthm,amssymb,graphicx,caption,subcaption}
\usepackage[margin=1.1in]{geometry}
\usepackage{tikz,enumitem}
\usepackage{hyperref}
\usepackage{esvect}
\usepackage{float}

\newtheorem{theorem}{Theorem}
\newtheorem{lemma}{Lemma}

\newtheorem{remark}{Remark}

\usepackage{soul}
\usepackage{subcaption}
\usetikzlibrary{calc}
\usepackage{pgf}
\usepackage{xfp}

\author{Gabriel Elvin, Hajrudin Fejzi\'c, and Youngsu Kim}

\begin{document}
\title{On Wegner's 8-Coloring Theorem for Planar Graphs of Maximum Degree Three}
\maketitle
\renewcommand{\thefootnote}{}
\renewcommand{\thefootnote}{\arabic{footnote}}
\begin{abstract}
We provide a simplified proof 
of the following special case of Wegner's conjecture: every planar graph of maximum degree at most three admits a distance-2 coloring with at most eight colors. 
Our main contribution is a significant simplification of the most technically challenging part of Wegner’s proof: the case involving the removal of a 5-cycle.

\end{abstract}

\section{Introduction and Overview}

A central result in graph theory is the \emph{Four Color Theorem}, which asserts that every planar graph can be colored using just four colors so that no two neighbors receive the same color. Proved with computer assistance in 1976, the Four Color Theorem has inspired a wide range of questions involving more restrictive coloring constraints.

One such problem, proposed by Wegner in 1977, involves coloring vertices so that even vertices that are two steps apart—that is, those that are neighbors or share a common neighbor—receive different colors \cite{Wegner1977}. This is known as a \emph{distance-2 coloring}, or sometimes a \emph{square coloring}. In this setting, each vertex must receive a different color not only from neighbors adjacent to it but also from any vertex connected to it by at most two edges apart.

Wegner conjectured in 1977 that every planar graph in which no vertex has more than three neighbors admits a distance-2 coloring using at most 7 colors. The conjecture remained open for decades and was finally resolved in 2016 by Hartke, Jahanbekam, and Thomas using the discharging method and extensive computer assistance \cite{Hartke}. It was also proved independently by Thomassen in 2018 \cite{Thomassen}. In the same article, Wegner also provided a proof that 8 colors always suffice and provided an example of a planar graph with 7 vertices that requires exactly 7 colors under distance-2 coloring constraints. See Figure~\ref{wegner}.

\begin{figure}[H]
\centering
\begin{tikzpicture}
\draw (90:2) -- (210:2) -- (330:2) -- cycle;
\draw (0:0) -- (90:2);
\draw (0:0) -- (210:2);
\draw (0:0) -- (330:2);
\draw[fill=white] (0:0) circle (0.1);
\draw[fill=white] (90:1) circle (0.1);
\draw[fill=white] (210:1) circle (0.1);
\draw[fill=white] (330:1) circle (0.1);
\draw[fill=white] (90:2) circle (0.1);
\draw[fill=white] (210:2) circle (0.1);
\draw[fill=white] (330:2) circle (0.1);
\end{tikzpicture}
\caption{Wegner's 7-vertex graph that requires 7 colors in any distance-2 coloring.}
\label{wegner}
\end{figure}
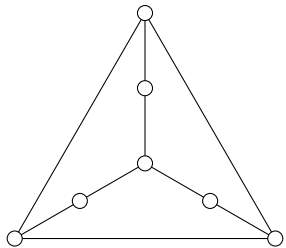

\bigskip

In this article, we revisit Wegner’s 8-color theorem and provide a novel treatment of his proof. Our approach avoids references to 3-connectedness, which Wegner used in combination with planarity. This is achieved through what we call the \emph{Inside-Outside Lemma}, a structural result that may be of independent interest. 
It allows us to formalize reasoning based on planar intuition while maintaining full combinatorial rigor.

Perhaps our main contribution is a significant simplification of the most technically challenging part of Wegner’s proof: the case involving the removal of a 5-cycle. Wegner distinguishes five configurations and identifies certain “critical” subcases. Some of his cases requires vertification of the outlines he provides, making the proof intricate to follow and less intuitive.

\bigskip

\subsection{Statement of the Theorem}
Throughout this paper, all graphs are assumed to be finite, undirected, and simple (that is, with no loops or multiple edges).

\emph{Planar graphs} are graphs that can be drawn in the plane without any edges crossing. Such a drawing is called a \emph{planar representation} or \emph{planar embedding} of a graph. These graphs are tightly constrained by geometry, and their structure often reveals useful patterns. 
In particular, connected planar graphs obey Euler’s formula: if $ V$, $ E$, and \( F \) are the numbers of vertices, edges, and faces of a planar representation of a graph, then $ V - E + F = 2 $.

 \emph{Distance-2 coloring} refers to a coloring in which no vertex shares a color with any of its neighbors of first or second order. That is, any pair of distinct vertices connected by a path of length one or two must have distinct colors.

To formalize this, we write \( v \sim_1 w \) if vertices \( v \) and \( w \) are adjacent (i.e., neighbors of first order), and \( v \sim_2 w \) if they are distinct and connected by a path of length two (i.e., neighbors of second order). We write \( v \sim w \) to mean that \( v \sim_1 w \) or \( v \sim_2 w \); in other words, \( v \) and \( w \) must receive different colors in any valid distance-2 coloring. 
In Figure~\ref{wegner}, $v_i\sim v_j$ for any pair $i\neq j$.

We now state Wegner’s theorem.
\begin{theorem}
Every planar graph in which no vertex has more than three neighbors admits a distance-2 coloring using at most 8 colors.
\end{theorem}

Following Wegner's approach, we proceed by contradiction. We assume that such a coloring is not always possible and let $G$  be a minimal counterexample—that is, a planar graph in which no vertex has more than three neighbors (a subcubic planar graph), which requires more than 8 colors in any distance-2 coloring, and which has the fewest possible number of vertices among all such graphs. Our goal is to show that no such minimal counterexample exists.

First, $G$ must be connected. If not, then by the minimality of $G$, each connected component could be colored independently using 8 or fewer colors, and the union of these colorings would yield a valid coloring of $G$, contradicting its minimality.

Next, we show  that $G$ cannot contain
a vertex $v$ which is of degree at most two or is part of a 3-cycle.
In each of these cases, we define a smaller graph $G'$  as shown in Figure~\ref{red_to_cubic} below.
In all three cases, $G'$ remains planar, sub-cubic, and has one less vertex than $G$, so by the minimality assumption, it admits a valid distance-2 coloring using at most 8 colors.

\begin{figure}[H]
\centering
\begin{subfigure}[b]{0.25\textwidth}
\centering
\begin{tikzpicture}[scale=1, every node/.style={circle,draw,fill=white,inner sep=2pt}]
\node (v) at (0,0) {v};
\node (w) at (2,0) {w};
\draw[dashed] (v) -- (w);
\end{tikzpicture}
\caption{$G'=G-v$}
\end{subfigure}
\begin{subfigure}[b]{0.32\textwidth}
\centering
\begin{tikzpicture}[scale=1, every node/.style={circle,draw,fill=white,inner sep=2pt}]
\node (v) at (0,0.5) {v};
\node (w) at (1,0) {w};
\node (u) at (1,1) {u};
\draw[dashed] (u) -- (v) -- (w);
\draw[thick] (u) -- (w);
\end{tikzpicture}
\caption{$G'=(G-v)+uw$}
\end{subfigure}
\qquad
\centering
\begin{subfigure}[b]{0.32\textwidth}
\begin{tikzpicture}[scale=1, every node/.style={circle,draw,fill=white,inner sep=2pt}]
\node (v) at (0.2,0.2) {v};
\node (u) at (1,1) {u};
\node (w) at (2,0) {w};
\node (a) at (-1,-0.5) {a};
\draw (u) -- (w);
\draw[dashed] (u) -- (v) -- (w);
\draw[dashed] (v) -- (a);
\draw[thick] (a) -- (w);
\end{tikzpicture}
\caption{$G'=(G-v)+aw$}
\end{subfigure}

\caption{}
\label{red_to_cubic}
\end{figure}

Moreover, in these configurations, the vertex $v$ has at most 3, 6, or 7 neighbors of first and second order, respectively. Thus, we can extend the coloring of $G'$ to a coloring of $G$ by assigning to $v$ any color not already used by its distance-2 neighbors.

Wegner observed, without proof, that every planar cubic graph must contain a cycle of length at most 5. For the sake of completeness, we now outline the standard argument.

Since $G$ is cubic, every vertex has degree 3, so the total number of edges satisfies $2E=3V$.
 Each face in a planar embedding is bounded by a closed walk, and each edge lies on the boundary of at most two faces. Letting $l_i$ denote the length of the boundary walk of the $i$-th face, we have: 
 \[\sum_{i=1}^F l_i=2E=3V.\] 
 
 Now, suppose for contradiction that every face boundary has length at least 6. Then: 
 
 \[\sum_{i=1}^F l_i\geq 6F\] so that $3V\geq 6F$. Substituting into Euler's formula we obtain \[2=V-E+F\leq V-\frac{3}{2}V+\frac{1}{2}V=0\] a contradiction. Therefore, some face must have boundary length at most 5.
 
Although face boundaries in a planar graph are closed walks and not necessarily cycles, any boundary walk of length at most 5 must contain a  cycle of length at most 5. If such a walk repeats a vertex, the repeated segment necessarily forms a smaller cycle\textemdash specifically of length at most 3\textemdash contradicting simplicity. Thus, some face boundary is a cycle of length 3, 4, or 5.

Since we already ruled out the possibility of a 3-cycle (see Figure~\ref{red_to_cubic} c), it follows that $G$ must contain a cycle of length  4 or 5.

Wegner relied on 3-connectedness together with planarity to control the placement of external neighbors around 4- and 5-cycles. While effective, this approach is technically involved and assumes a background in structural graph theory. Instead, we adopt a more geometric and intuitive strategy based on planarity alone. We do this with the help of the following lemma.

\begin{lemma}[Inside-Outside Lemma]\label{io}
Let \( G \) be a cubic planar graph, and let \( C \) be a cycle in \( G \) of length 3, 4, or 5. 
Fix a planar representation of \( G \). 
Let \( H_0 \) denote the induced subgraph of \( G \) by the vertices of \( C \) and inside \( C \), 
and let \( H_1 \) denote the induced subgraph by the vertices of \( C \) and outside \( C \).
If \( H_0 \) and \( H_1 \) can be distance-2 colored using \( m \) colors, 
then \( G \) can be distance-2 colored using \( m \) colors.
\end{lemma}

\begin{proof}
Let \( G_i \) be the induced subgraphs of \( H_i \) after deleting the vertices of \( C \) for \( i = 0, 1\). 
Suppose \( u \in G_0 \) and \( w \in G_1 \). If \( u \sim_2 w \) in \( G \), then there must exist a vertex \( v \in C \) such that \( u \sim_1 v \sim_1 w \); see Figure~\ref{uw}. 
But this would imply that \( v \) has at least four first neighbors\textemdash namely, its two neighbors on the cycle, along with \( u \) and \( w \). 
Since \( G \) is cubic, this is impossible.

Therefore, there are no distance-2 constraints between \( G_0 \) and \( G_1 \) other than those involving vertices in \( C \). 
Since all vertices on \( C \) must receive distinct colors in any distance-2 coloring, we can independently color \( H_0 \) and \( H_1 \) using \( m \) colors, 
and then match the colorings along \( C \). 
This yields a distance-2 coloring of the full graph \( G \).
\end{proof}

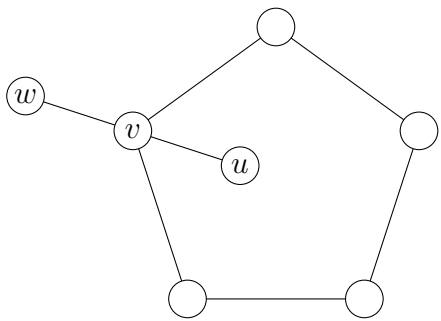
\begin{figure}[H]
\centering
\def\c{5}
\def\a{\fpeval{360/\c}}
\begin{tikzpicture}
\draw (162:0.5) -- (162:3.5);
\foreach \t in {1,...,\c}{
\draw (\a*\t + 90 : 2) -- (\a*\t + 90 + \a : 2);
}
\foreach \t in {1,...,\c}{
\draw[fill=white] (\a*\t + 90: 2) circle (0.25);
}
\draw (162:2) node{$v$};
\draw[fill=white] (162:0.5) node{$u$} circle (0.25);
\draw[fill=white] (162:3.5) node{$w$} circle (0.25);
\end{tikzpicture}
\caption{A 5-cycle \( C \) with interior vertex \( u \in G_0 \) and exterior vertex \( w \in G_1 \), both connected to \( v \in C \)}
\label{uw}
\end{figure}

\begin{remark}
The Inside-Outside Lemma implicitly uses the Jordan Curve Theorem, which guarantees that any cycle in a planar graph separates the plane into an interior and an exterior region, with the cycle as their common boundary. Furthermore, if all other vertices lie inside the cycle in a planar embedding, we may invert the planar presentation (e.g., via stereographic projection) so that they lie outside instead. 

Thus, in our planar representation of the last two cases\textemdash where the minimal cycle \( C \) is a 4-cycle or 5-cycle\textemdash we may assume, without loss of generality, that the interior of \( C \) is empty.
\end{remark}

In the case of the 5-cycle \( (x_1,x_2,x_3,x_4,x_5) \), we now explain further assumptions about the structure of the exterior neighbors, which also apply (by analogy) in the 4-cycle case. Since \( G \) is cubic, each vertex \( x_1, x_2, x_3, x_4, x_5 \) on the cycle has a third neighbor outside the cycle, which we denote \( y_1, y_2, y_3, y_4, y_5 \), respectively. 
See Figure~\ref{5cycle} below. 

\def\c{5}
\def\a{\fpeval{360/\c}}
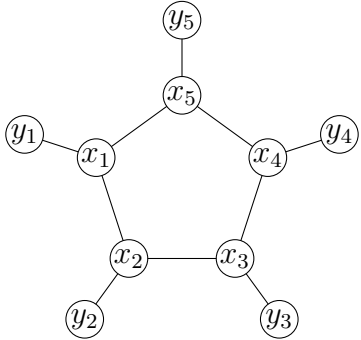
\begin{figure}[H]
\centering
\begin{tikzpicture}[scale=1]
\foreach \t in {1,...,\c}{
\draw (\a*\t + 90 : 1.2) -- (\a*\t + 90 + \a : 1.2);
}
\foreach \t in {1,...,\c}{
\draw (\a*\t + 90: 1.2) -- (\a*\t + 90: 2.2);
\draw[fill=white] (\a*\t + 90: 2.2) node{$y_{\t}$} circle (0.25);
\draw[fill=white] (\a*\t + 90: 1.2) node{$x_{\t}$} circle (0.25);
}
\end{tikzpicture}
\caption{5-cycle with exterior neighbors}
\label{5cycle}
\end{figure}

It is not immediately obvious as the drawing suggest, that these five exterior neighbors are distinct. To justify this, we proceed as follows: by the minimality of the 5-cycle, \( y_2 \) and \( y_5 \) must be distinct from \( y_1 \), since otherwise the graph would contain a 3-cycle. 
Next, we apply {the Inside-Outside} (Lemma~\ref{io}) to conclude that \( y_3 \) and \( y_4 \) must also be distinct from \( y_1 \). 
For example, if $y_4=y_1$, 
then using planarity one side of the enclosing  4-cycle \( (y_1,x_1,x_5,x_4,y_1) \) would enclose either \( x_2 \) and \( x_3 \) inside and $y_5$ outside or the other way around. 
Thus, we could apply  Lemma~\ref{io} to 4-cycle \( (y_1,x_1,x_5,x_4,y_1) \) and complete the coloring of \( x_5 \), with 8 colors or less.  By symmetry, we may therefore assume that all five neighbors \( y_1, y_2, y_3, y_4, y_5 \) are pairwise distinct.

In our proofs of the 4-cycle and 5-cycle cases, we will need to assume that \( y_1 \not\sim_1 y_3 \) and  \( y_1 \not\sim_1 y_4 \) respectively.   For the 5-cycle case, suppose instead that the edge \(y_1y_4\)  exists. Then the resulting 5-cycle \( C = (y_1,x_1,x_5,x_4,y_4,y_1) \) (see Figure~\ref{5cycle}) would enclose either \( x_2 \) and \( x_3 \) on the interior and \( y_5 \) on the exterior, or the reverse.
 
In either case, the separation induced by \( C \) would allow us to apply Lemma~\ref{io}, coloring the inside and outside regions independently and matching the coloring along the cycle. This would yield a valid 8-coloring of $G$, contradicting its minimality as a counterexample. The 4-cycle case follows similarly.

\section{The 4-Cycle case}

Suppose the minimal cycle in \( G \) is a 4-cycle \( (x_1,x_2,x_3,x_4,x_1) \) with external neighbors \( y_1, y_2, y_3, y_4 \) adjacent to \( x_1, x_2, x_3, x_4 \), respectively, as shown in Figure~\ref{4-cycle}. 
Define \( G'=G-\{x_1,x_2,x_3,x_4\}+y_1y_3 \).
Since  the new edge $y_1y_3$ can be drawn along the trace of the path \( y_1 \!-\! x_1 \!-\! x_4 \!-\! x_3 \!-\! y_3 \) from the original planar embedding of $G$, the resulting graph $G'$ remains planar.

\begin{figure}[H]
\centering
\begin{subfigure}[b]{0.45\textwidth}
    \centering
\begin{tikzpicture}
\foreach \t in {0,1,2,3}{
\draw[dashed] (45+90*\t:1.25) -- (135+90*\t:1.25);
\draw[dashed] (45+90*\t:1.25) -- (45+90*\t:2.125);
\draw[fill=white] (45+90*\t:1.25) circle (0.25);
\draw[fill=white] (45+90*\t:2.125) circle (0.25);
}
\draw[fill=white] (45:1.25) node{$x_4$} circle (0.25);
\draw (135:1.25) node{$x_1$};
\draw (-135:1.25) node{$x_2$};
\draw (-45:1.25) node{$x_3$};
\draw[fill=white] (45:2.125) node{$y_4$};
\draw (135:2.125) node{$y_1$};
\draw (-135:2.125) node{$y_2$};
\draw (-45:2.125) node{$y_3$};
\draw[thick] (130:1.95) .. controls (45:1) .. (-40:1.95);
\end{tikzpicture}
\caption{A 4-cycle with the exterior neighbors}
\label{4-cycle}
\end{subfigure}
\qquad
\begin{subfigure}[b]{0.45\textwidth}
\centering
\def\c{5}
\def\a{\fpeval{360/\c}}
\centering
\begin{tikzpicture}[scale=1]
\draw[thick] (\a + 90 : 2.2) .. controls (90:2) .. (4*\a + 90 : 2.2);
\foreach \t in {1,...,\c}{
\draw[dashed] (\a*\t + 90 : 1.2) -- (\a*\t + 90 + \a : 1.2);
}
\foreach \t in {1,...,\c}{
\draw[dashed] (\a*\t + 90: 1.2) -- (\a*\t + 90: 2.2);
\draw[fill=white] (\a*\t + 90: 2.2) node{$y_{\t}$} circle (0.25);
\draw[fill=white] (\a*\t + 90: 1.2) node{$x_{\t}$} circle (0.25);
}
\end{tikzpicture}
\caption{A 5-cycle with the exterior neighbors}
\label{5-cycle}
\end{subfigure}
\caption{}
\end{figure}

By the minimality of \( G \), the reduced graph \( G' \) admits a distance-2 coloring using at most 8 colors. 
We extend coloring of \( G'\) to \( G \).

Since \( y_1 \) and \( y_3 \) are adjacent in $G'$, they have different colors. By symmetry, we may assume  \( y_3 \) and \( y_4 \) are also  differently colored. 
Uncolor \( y_2 \) and color \( x_1 \) with the color of \( y_3 \). 
Now \( y_2 \) has at most nine distance-2 neighbors, including uncolored  \( x_2 \) and \( x_3 \), so it can be recolored.

The remaining vertices \( x_2, x_3, \) and \( x_4 \) lie on a 4-cycle and each has at most eight distance-2 neighbors, including two uncolored ones, so each has at least two available colors. Moreover, \( x_4 \) has at least three available colors, since both \( x_1\) and \( y_3 \) share color. Thus, we can complete  the 8-coloring of \( G \), contradicting minimality. 

\section{The 5-Cycle case}
Suppose the minimal cycle in \( G \) is a 5-cycle \( (x_1,x_2,x_3,x_4,x_5,x_1) \) with external neighbors \( y_1, y_2, y_3, y_4 ,y_5\) adjacent to \( x_1, x_2, x_3, x_4,y_5 \), respectively, as shown in Figure~\ref{5-cycle}. 
Define \( G'=G-\{x_1,x_2,x_3,x_4,x_5\}+y_1y_4 \).
Since  the new edge $y_1y_4$ can be drawn along the trace of the path \( y_1 \!-\! x_1 \!-\! x_5-\! x_4 \!-\! y_4 \) from the original planar embedding of $G$, the resulting graph $G'$ remains planar.

By the minimality of \( G \), the reduced graph \( G' \) admits a distance-2 coloring using at most 8 colors. 
We extend coloring of \( G'\) to \( G \).

Since \( y_1 \) and \( y_4 \) are adjacent in $G'$, they have different colors. By symmetry, we may assume  \( y_4 \) and \( y_5 \) are also  differently colored. 

We begin by uncoloring $y_2$ and $y_3$, and assigning to $x_1$ the color of $y_4$. Then $y_2$, which has at most 9 distance-2  neighbors (including the yet to be colored $x_2$ and $x_3$), can be recolored.

Next, we attempt to color $x_4$ with the same color as $y_1$. If this is not possible, then $y_1$ and $y_5$ must share a color, and we assign to $x_4$ any admissible color.

Now we color $y_3$, which like $y_2$, has at most 9 distance-2 neighbors, including two uncolored vertices, so it has at least one available color.

We finish by coloring the remaining vertices $x_2$, $x_3$, $x_5$. Each has at most 9 distance-2 neighbors, including two uncolored ones, and thus has at least one admissible color. In particular, $x_3$, which sees both $x_1$ and $y_4$ (now sharing a color), has at least two admissible colors.

Finally, consider $x_5$. Its neighbors include $x_1$ and $y_4$, which share a color. If $x_4$ was colored to match $y_1$, then $x_4$ and $y_1$ are another pair of neighbors with the same color. Otherwise, $y_1$ and $y_5$ share a color. In both cases, $x_5$ has at least two pairs of neighbors sharing a color, so among its 7 colored neighbors, at most 5 distinct colors appear. Thus, $x_5$ has at least 3 admissible colors.\\

\noindent This completes a valid 8-coloring of $G$, contradicting the assumption that $G$ was a minimal counterexample.

\end{document}